\documentclass[11pt]{statsoc}

\usepackage{geometry}
\usepackage{natbib}
\usepackage{graphics,graphicx}
\usepackage{bbm}
\usepackage{amsmath}
\usepackage{amssymb}
\usepackage[all]{xy}

\usepackage{verbatim}
\usepackage{bbm}

\usepackage{graphicx,amscd,amsmath,amssymb,amsfonts,verbatim}
\usepackage{mathrsfs,graphics}
\usepackage{bbm}
\usepackage[all]{xy}
\usepackage{xcolor}
\usepackage{bm}
\usepackage{bbm}

\DeclareMathAlphabet{\mathpzc}{OT1}{pzc}{m}{it}

\def\E{\mathbb{E}}
\def\E{\mathbbmss{E}}

\def\transpose{\prime}

\def\bx{\bm{x}}

\def\bW{\bm{W}}

\def\bW{\bm{W}}
\def\bX{\bm{X}}

\def\CiD{\stackrel{d}{\longrightarrow}}
\def\CiP{\stackrel{p}{\longrightarrow}}

\newtheorem{theorem}{Theorem}
\newtheorem{lemma}[theorem]{Lemma}

\renewcommand{\qed}{\nobreak \ifvmode \relax \else
      \ifdim\lastskip<1.5em \hskip-\lastskip
      \hskip1.5em plus0em minus0.5em \fi \nobreak
      \vrule height0.75em width0.5em depth0.25em\fi}

\def\Pr{\text{Pr}}

\def\Ell{\mathcal{L}}

\makeatletter
\renewcommand\appendix{%
\global\@topnum\z@
\ifdim\lastskip>0pt\vskip-\lastskip\fi
  \setcounter{section}{0}%
  \setcounter{subsection}{0}%
  \renewcommand\thesection{Appendix \@Alph\c@section}
  \renewcommand{\theequation}{\thesection.\arabic{equation}}
  \setcounter{equation}{0}
\vskip16pt plus 4pt minus 2pt
\noindent{\bfseries\center}
\nobreak}

\title{Variable Selection in Causal Inference Using Penalization}
\author[ ] {Ashkan Ertefaie \footnote{\tiny{\it Address for correspondence:} Department
of Statistics, University of Michigan, Ann Arbor, 48105, Michigan, USA. E-mail: ertefaie@umich.edu\\}}
 \address{University of Michigan,
Ann Arbor, USA. \smallskip} \email{}

\author[ ] {Masoud Asgharian}

\author[ ] {David A. Stephens}
\address{McGill University, Montreal, Canada.}

\begin{document}

\begin{abstract}
In the causal adjustment setting, variable selection techniques based on either the outcome or treatment allocation model can result in the omission of confounders or the inclusion of spurious variables in the propensity score.  We propose a variable selection method based on a penalized likelihood which considers the response and treatment assignment models simultaneously. The proposed method facilitates confounder selection in high-dimensional settings. We show that under some conditions our method attains the oracle property. The selected variables are used to form a double robust regression estimator of the treatment effect.
Simulation results are presented and  economic growth data are analyzed.


\end{abstract}

\section{Introduction} \label{intro}

In the analysis of observational data, when attempting to establish the magnitude of the causal effect of treatment (or exposure) in the presence of confounding, the practitioner is faced with certain modeling decisions that facilitate estimation.  Should one take the parametric approach, at least one of two statistical models must be proposed; (i) the \textit{conditional mean model} that models the expected response as a function of predictors, and (ii) the \textit{treatment allocation model} that describes the mechanism via which treatment is allocated to (or, at least, received by) individuals in the study, again as a function of the predictors \citep{rosenbaum1983central, robins2000marginal}.

Predictors that appear in both mechanisms (i) and (ii) are termed \textit{confounders}, and their omission from model (ii) is typically regarded as a serious error, as it leads to inconsistent estimators of the treatment effect. Thus practitioners usually adopt a conservative approach, and attempt to ensure that they do not omit confounders by fitting a richly parameterized treatment allocation model. 
The conservative approach, however, can lead to predictors of treatment allocation, but not response, being included in the treatment allocation model. The inclusion of such ``spurious" variables in model (ii) is usually regarded as harmless.  However, the typical reported forfeit for this conservatism is inflation of variance of the effect estimator \citep{greenland2008invited, schisterman2009overadjustment}.  This problem also applies to the conditional mean model, but is in practice less problematic, as practitioners seem to be more concerned with bias removal, and therefore more liable to introduce the spurious variables in model (ii).  Little formal guidance as to how the practitioner should act in this setting has been provided.

As has been conjectured and studied in simulation by \cite{brookhart2006variable}, it is plausible that judicious variable selection may lead to appreciable efficiency gains.  However, confounder selection methods based on either just the treatment assignment model, or just the response model, may fail to account for non-ignorable confounders which barely predict the treatment or the response, respectively \citep{crainiceanu2008adjustment}.  In this manuscript, we use the term \textit{weak confounder} for these variables. \cite{vansteelandt2010model} shows that confounder selection procedures based on AIC and BIC can be sub-optimal and introduce a method based on the focused information criterion (FIC) which targets the treatment effect by minimizing a prediction mean square error (see also the cross-validation method of \cite{brookhart2006semiparametric}). \cite{van2007super} introduces a \textit{Super Learner} estimator which is computed by selecting a candidate from a set of estimators obtained from different models using a cross-validation risk \citep{van2004cross, sinisi2007super}. 

\cite{van2010collaborative}  selects the sufficient and minimal variables necessary for the propensity score model to estimate an unbiased causal effect by inspecting the efficient influence function \citep{porter2011relative}.  \cite{de2011covariate} discusses the variance inflation caused by adding the spurious variables in the model and show that it may cause bias as well. Under some assumptions, they also characterize the minimal set of covariates needed for consistent estimation of the treatment effect. Bayesian adjustment for confounding (BAC) is another method introduced by \cite{wang2012bayesian}. BAC specifies a prior distribution for a set of  possible models which includes a dependence parameter, $w\in [1,\infty]$, representing the odds of including a variable in the outcome model given that the same variable is in the treatment mechanism model.  Assuming that we {\it know} a priori that all the predictors of the treatment are in fact confounders, then $w$ can be set to $\infty$ \citep{crainiceanu2008adjustment, zigler2013model}. However, in practice,  none of these methods can be used in high-dimensional settings where the number of covariates are larger than sample size.

It is known that {\it asymptotically} penalizing the conditional outcome model, given treatment and covariates, results in a valid variable selection strategy in causal inference. However, for small to moderate sample sizes this may result in missing weak non-ignorable confounders, which barely predict the outcome but strongly predict the treatment mechanism. The objective of this manuscript is to improve the small sample performance of the outcome penalization strategy while maintaining its asymptotic performance (Table \ref{tab:largep2}). We present a covariate selection procedure which facilitates the estimation of the treatment effect in the {\it high-dimensional} cases. We parametrize the conditional joint likelihood of the outcome and treatment given covariates such that penalizing this joint likelihood has the ability  to select even weak confounders, i.e., confounders which are non-ignorable even if they are barely associated with the outcome or treatment mechanism. This likelihood is just used to identify the set of important covariates, i.e., non-ignorable confounders and predictors of outcome,  and, in general, the estimated parameters do not have any causal interpretation. We derive the asymptomatic properties of the maximum penalized likelihood estimator using a method that does not require the second derivative of the joint density function.  We utilize the selected covariates to estimate the causal effect of interest using our proposed doubly robust estimator.

We restrict our attention to the \textit{unmediated} causal effect (where the effect of exposure on outcome is not mediated by an intermediate variable); in the presence of mediation, direct and indirect effects may not in general be identifiable \citep{robins1992identifiability, petersen2006estimation, robins2010identification, hafeman2010alternative}.


\section{Preliminaries \& Notation} \label{sec:prelim}
Let $Y(d)$ denote the (potential) response to treatment $d$, and let $D$ denote the treatment received. The observed response, $Y$, is defined as $DY(1)+(1-D)Y(0)$. We will assume three types of predictors:

\begin{itemize}

\item[(I)] \textit{treatment predictors} ($X_1$), which are related to treatment and not to outcome.

\item[(II)] \textit{confounders} ($X_2$), which are related to both outcome and treatment.

\item[(III)] \textit{outcome predictors} ($X_3$), which are related to outcome and not to treatment;
\end{itemize}
see the directed acyclic graph (DAG) in Figure \ref{fig:2conf}.
\begin{figure}[t]
\vspace{0.4 in}
\caption{ Covariate types: Type-I: $X_1$, Type-II: $X_2$ and Type-III: $X_3$.}
\vspace{0.4 in}
\centerline{
\xymatrix{
 X_2\ar[r]  \ar@/^2pc/[rr] & D \ar[r] & Y\\
  X_1 \ar[ur] & X_3 \ar[ur]\\
  }
}
\label{fig:2conf}
\end{figure}
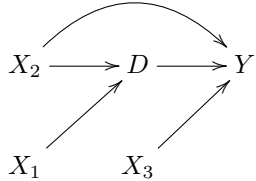
We restrict our attention here to the situation where each predictor can be classified into one of these three types, and to single time-point studies. In addition, as is usual, we will make the assumption of \textit{no unmeasured confounders}, that is, that treatment received $D$ and potential response to treatment $d$, $Y(d)$, are independent, given the measured predictors $X$.  In any practical situation, to facilitate causal inference, the analyst must make an assessment as to the structural nature of the relationships between the variables encoded by the DAG in Figure \ref{fig:2conf}.

\subsection{The Propensity Score for binary treatments}

The propensity score, $\pi$, for binary treatment $D$ is defined as $\pi(x) =
\Pr(D=1|x)$, where $x$ is a $p$-dimensional vector of (all) covariates.
\cite{rosenbaum1983central} show that $\pi$ is the coarsest function of
covariates that exhibits the balancing property, that is, $D \perp X
| \pi $. As a consequence, the causal effect $\mu = \E[Y(1)-Y(0)]$ can be
computed by iterated expectation
\begin{equation}
\mu = \E_{X}[\E\{Y(1)|X\}-\E\{Y(0)|X\}] = \E_{\pi}[\E\{Y(1)|\pi\}-\E\{Y(0)|\pi\}],
\label{eq:ATE}
\end{equation}
where $\E_{\pi}$ denotes the expectation with respect to the
distribution of $\pi$. For more details see \cite{rubin2008objective} and
\cite{rosenbaum2010causal}.


\medskip

\noindent \textbf{Remark 1:} In the standard formulation of the propensity score, no distinction is made between our three types of covariates.  Note that, however, for consistent estimation of $\mu$, \textit{it is not necessary to
balance on covariates that are not confounders}. Covariates $X_1$ that predict $D$ but not $Y$ may be unbalanced in treated and untreated groups, but will not affect the estimation of the effect of $D$ on $Y$, as $D$ will be conditioned
upon, thereby blocking any effect of $X_1$ \citep{de2011covariate}.  Covariates $X_3$ are unrelated to
$D$, so will by assumption be in balance in treated and untreated groups in the population. Therefore, the propensity score need only be constructed from confounding variables $X_2$; in this case, it is easy to see that the
propensity score, $\pi_{2}=\pi_2(x_2) $, say, is a balancing score in the sense that $D\perp X_2 \: | \: \pi_{2}$:  we have $\Pr(D=1\:|\:\pi_{2}(x_2)=t,X_2=x_2) = t =
\Pr(D=1\:|\:\pi_{2}(x_2)=t)$, independent of $x_2$, in the usual way.  Then, in the presence of outcome predictors $X_3$ of $Y$, the sequel to equation \eqref{eq:ATE}
takes the form
\begin{eqnarray}
\mu = \E[Y(1)-Y(0)] & = & \E_{X_2,X_3}[\E\{Y(1)|X_2,X_3\}-\E\{Y(0)|X_2,X_3\}] \nonumber\\[6pt]
& = & \E_{\pi_{2},X_3}[\E\{Y(1)|\pi_{2},X_3\}-\E\{Y(0)|\pi_{2},X_3\}].
\label{eq:ATE2}
\end{eqnarray}

\noindent \textbf{Remark 2:} Inclusion of covariates that are just related to the outcome in the propensity score model increases the covariance between the fitted $\pi$ and $Y$, decreases the variance of the estimated causal effect, in line with the simulation of \cite{brookhart2006variable}.  

\subsection{Penalized Estimation}


In a given parametric model, if $\eta$ is a $r$-dimensional regression coefficient,  $p_{\lambda}(.)$ is a penalty function and $l_m(\eta)$ is the negative log-likelihood, the maximum penalized likelihood (MPL) estimator $\widehat \eta_{ml}$ is defined as
\begin{align*}
\widehat \eta_{ml}=\arg \min_{\eta} \left [l_m(\eta)+n\sum_{j=1}^r p_{\lambda}(|\eta_j|) \right ].
\end{align*}
MPL estimators are shrinkage estimators, and as such, they have more bias, though less variation than unpenalized analogues. Commonly used penalty functions include LASSO \citep{tibshirani1996regression}, SCAD
\citep{fan2001variable}, EN \citep{zou2005regularization} and HARD \citep{antoniadis1997wavelets}. 

The remainder of this paper is organized as follows. Section \ref{sec:penalized} presents our two step variable selection and estimation procedure; we establish its theoretical properties.  The performance of the proposed method is studied via simulation in Section \ref{sec:simulation}. We analyze a real data set in Section \ref{sec:application}, and Section \ref{sec:conclude} contains concluding
remarks.  All the proofs are relegated to the Appendix.

\section{Penalization and Treatment Effect Estimation} \label{sec:penalized}

In this section, we develop the methodology which facilitates the estimation of the treatment effect in high-dimensional cases.  We separate the covariate selection and the treatment effect estimation procedure. First, we present a reparametrized penalized likelihood which is used to identify the important covariates, and establish the theoretical properties of the resulting MPL estimators. Note that since the likelihood is reparametrized the MPL estimators do not have any causal interpretation. Second, the treatment effect estimation is performed using our doubly robust estimator with the selected covariates in the previous step.

\subsection{Likelihood construction}

Consider the parametric likelihood $\Ell(\eta;y,d,x)$ proportional to
\begin{align}
\prod_{i=1}^n  f(y_i|d_i,g(x_i;\alpha), \beta ) P(D=1|h(x_i,\alpha))  ^{d_i}
P(D=0|h(x_i,\alpha))^{1-d_i},
\label{eq:Hahn}
\end{align}
where $\beta$ is an $r_1$-dimensional vector parametrizing the association between the outcome and the treatment and $\alpha$ is an $r_2$-dimensional vector containing parameters that appear in the model for $Y(d)|X$ and $D|X$. The functions $g()$ and $h()$ used in our joint likelihood have the same form as one would use when modeling the outcome model and treatment mechanism separately. For example, assuming  linear working models, $g(\bx;\alpha)=\sum_{j=1}^{r2} \alpha_j x_j$ and $h(\bx;\alpha)=\sum_{j=1}^{r2} \alpha_j x_j$. Note that for each $j$, the parameter $\alpha_j$ corresponding to $x_j$ is the same in both models. This is why we call (\ref{eq:Hahn}) a reparametrized likelihood. We explain the rational behind this reparametrization in section 3.2. 

Since our goal is to select the minimal set of covariates necessary  for a consistent estimation of the causal effect, we impose a penalty on the parameters $\alpha$ only; there is no penalization of the $\beta$ parameters. The penalized pseudo-density for $z_i=(y_i,d_i,x_i)$ is $f_p(z_i,\eta)
=f(z_i;\eta) f(\alpha)$, where $f(z_i;\eta)$ is the joint density used in (\ref{eq:Hahn}) and $f(\alpha)=\exp\{-p_{\lambda_n}(\alpha)\}$. Accordingly, the
MPL estimator, $\widehat \eta$, can be defined by
\[
\widehat \eta =  \arg \sup_{\eta}\prod_{i=1}^n f_p(z_i; \eta) = \arg \sup_{\eta} \sum_{i=1}^n \log f_p(z_i; \eta).
\]
Note the joint density (\ref{eq:Hahn}) is a misspecification of the true data density. As such, the corresponding penalized likelihood just  checks whether $\alpha_k=0$ for $k=1,...,r_2$, and is not for other estimation purposes. This is discussed in detail in the following subsection.

\subsection{Avoiding omission of confounders during selection}

Standard variable selection techniques based on the conditional outcome/treatment model have the tendency to omit important confounders by ignoring covariates that are weakly associated with the outcome/treatment but strongly associated with treatment/outcome \citep{vansteelandt2010model}. However, our likelihood parametrization in (\ref{eq:Hahn}),  which has the parameter $\alpha$ in both response and propensity score models, allows us to select such weak confounders. More specifically, our parametrization gives each covariate two chances to appear in the model; once in the response model and once in the treatment allocation model and thus considers both the covariate-exposure and the covariate-outcome association.  Our reparametrization has a drawback of setting $\alpha=0$ if the tradeoff between the value of the coefficient in the two parts of the likelihood somehow cancel out. In other words, when  the association parameter of a variable with the outcome and treatment have opposite signs, then for  particular association values, the reparametrized likelihood sets the parameter corresponding to the variable to zero.  However, in \ref{app:measure}, we show that this particular data generating low has zero measure.

Our proposed parametrization, however, has another drawback that needs to be taken care of. Figure \ref{fig:mislike-perf} shows that this strategy sets $\alpha \neq 0$ if a covariate is related to either the outcome or treatment. This figure presents a case where there is just one covariate and the coefficient of this covariate in outcome and treatment models are $1/\sqrt n$ and $0.3$, respectively, where $n$ is the sample size. As it is expected, the estimated parameter $\alpha$ corresponding to this covariate does not converge to zero when estimated using the likelihood (\ref{eq:Hahn}) as sample size increases.  Hence, our parametrization gives an equal chance to Type-I and Type-III covariates for selection as key covariates. This may result in over-representing the Type-I variables which is against our goal of keeping variables which are either predictors of the response or non-ignorable confounders. To deal with this problem, we introduce the {\it boosting} parameter $\nu$ which boosts covariates Type-III relative to Type-I. The boosting parameter can be defined as $\nu= \frac{1}{| \tilde \alpha_Y|(1+| \tilde \alpha_D|)}$, where $\tilde \alpha_Y$ and $\tilde \alpha_D$ are the least squares (or ridge) estimate of the parameters in the response and treatment models, respectively. Our
penalty function is proportional to the boosting parameter,
$$p_{\lambda_n}(.)=\nu p_{\lambda_n}^*(.), $$ where $p_{\lambda_n}^*(.)$ is a
conventional penalty function. Therefore, the magnitude of the penalty on each
parameter is proportional to its contribution to the response
model.  Note that as $\tilde \alpha_Y \rightarrow 0$, our penalty function puts more penalty on  the parameters while considering the covariate-treatment association. For example, when a covariate barely predicts the outcome and treatment, our proposed penalty function imposes a stronger penalty on the parameter compared to a case where a covariate barely predicts the outcome and is strongly related to treatment.  For example, when $p_{\lambda_n}^*(.) $ is lasso, our penalty is $p_{\lambda_n}(|\alpha_j|)= \lambda_n \nu_j |\alpha_j| $.
A similar argument can be found in the adaptive LASSO \citep{zou2006adaptive}.

\begin{figure}[t]
\centering
\includegraphics[scale=.6]{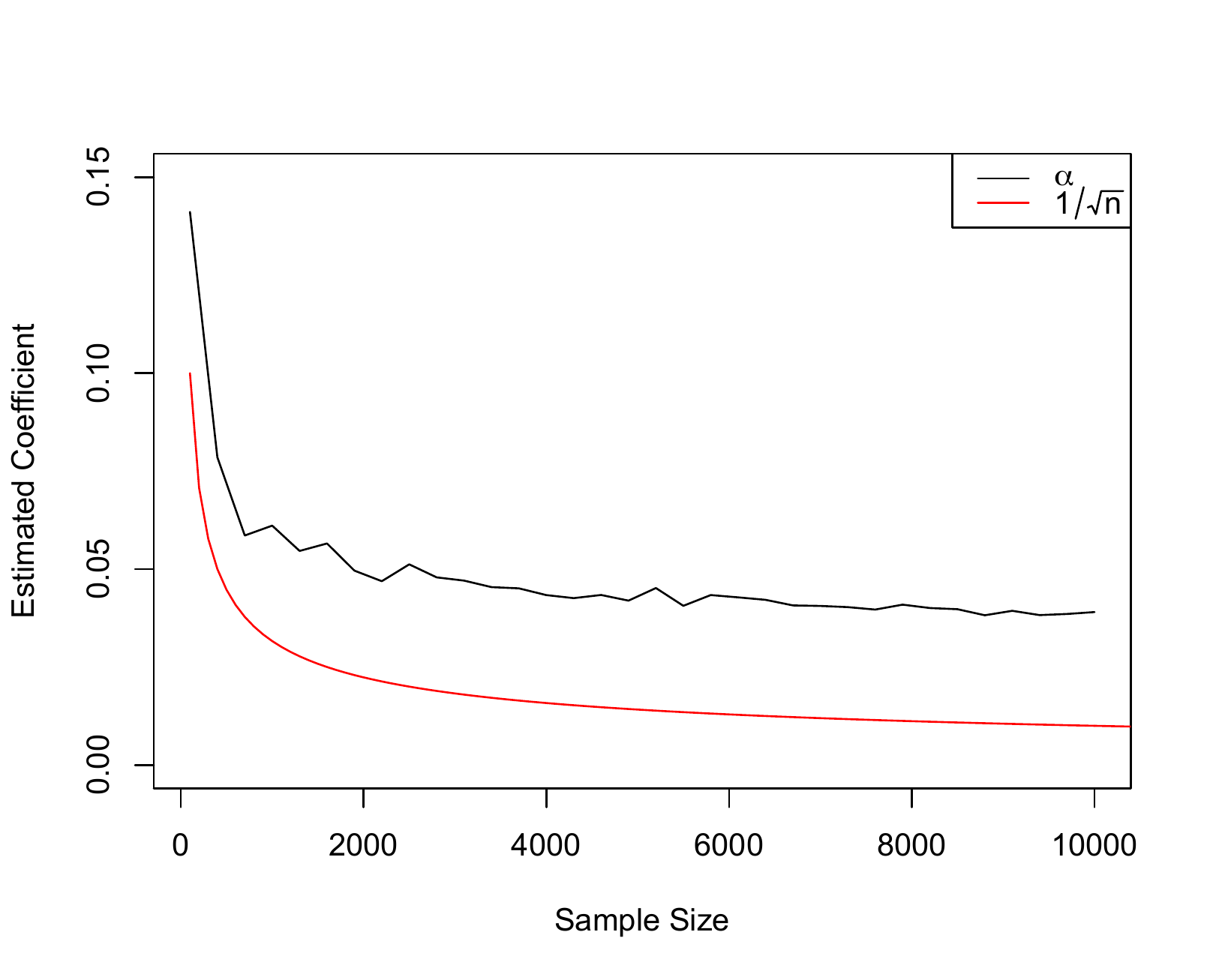}
\caption{{Performance of the misspecified likelihood for different sample sizes $n$. Red and black  lines are $1/\sqrt n$ and the estimated coefficient $\alpha$ using the reparametrized likelihood.  }} 
\label{fig:mislike-perf}
\end{figure} 

\subsection{Main Theorems} \label{theorems}


The following conditions guarantee a consistent
penalized estimating procedure for the parameter $\eta$ with respect to the likelihood $(\ref{eq:Hahn})$ which sets the small coefficients to zero for
covariate selection.
\begin{itemize}

\item[P1.]     For all $n$, $p_{\lambda_{n}}(0)=0$ and $p_{\lambda_{n}}(\alpha)$ is non-negative, symmetric about 0 and it is non-decreasing on both $\mathcal{R}^{+}$ and $\mathcal{R}^{-}$, i.e. on positive and negative half line. Moreover, it is twice differentiable with derivatives $p_{\lambda_{n}}^{'}(\alpha)$ and $p_{\lambda_{n}}^{''}(\alpha)$ exist everywhere except at $\alpha = 0$.

\item[P2.] As $n \rightarrow \infty$, $\max_{\alpha\neq0}[ p''_{\lambda_n}(\alpha) ] \rightarrow 0$ and
$\max_{\alpha\neq0}[ \sqrt np'_{\lambda_n}(\alpha) ] \rightarrow 0$.

\item[P3.] For $N_n \equiv (0,B_n)$, $\text{lim $\inf\limits_{\alpha \in \:
    N_n}$}  p'_{\lambda_n}(\alpha) = \infty$, where $B_n \rightarrow 0$ as $n \rightarrow \infty$.
\end{itemize}
Assumption P1 is used to prove Theorem \ref{th:PMLEexists} given in \ref{app:PMLEexists}, while P2 prevents
the $j$th element of the penalized likelihood from being dominated by the
penalty function since it vanishes when $n \rightarrow \infty$. If
$\alpha_{j}=0$, condition P3 allows the penalty function to dominate the
penalized likelihood which leads to the sparsity property.


Suppose the  $r$-dimensional vector of parameters $\eta_0=(\eta_{01},\eta_{02}=0)$ is the true values of the parameter $\eta$, such that
$\eta_{02}=(\eta_j) = 0$ for $j=s+1,...,r$; $s$ denotes the true number of
predictors present in the model ({\it exact sparsity} assumption).  Note that since there is no penalty on the
$\beta$s, $\eta_{02}$ consists of those $\alpha$ that should be shrunk to zero
($\alpha_j=0$ for $j=s',...,r_2$).  Let $\widehat \eta=(\widehat \eta_1,\widehat \eta_2)$ be the vector of MPL estimators  corresponding to \eqref{eq:Hahn}.

Theorem \ref{th:PMLEexists} in \ref{app:PMLEexists} establishes the existence of the consistent penalized maximum likelihood estimator with respect to the
joint likelihood (\ref{eq:Hahn}) under standard regularity conditions (\cite{ibragimov1981statistical}) given as C1-C4 in \ref{app:conditions}.


The next theorem proves the sparsity and asymptotic normality of  the MPL estimators. Let $I(\eta)$
be the Fisher information matrix derived from the constructed likelihood.


\medskip
\begin{theorem} \textbf{(Oracle properties)}
Suppose assumptions C1-C4 and P1-P3 are fulfilled and further $\det [I(\eta)] \neq 0$ for $\eta \in \Xi$. Then
\begin{enumerate}
\item[(a)] $\Pr(\widehat \eta_2=0)\rightarrow 1$ as $n\rightarrow \infty$
\item[] Under additional assumption  $C5$, 
\label{th:norm}
\item[(b)] $\sqrt n(\widehat \eta_{01}- \eta_{01}) \CiD N(0,I^{-1}(\eta_{01})),$
\end{enumerate}
where $\eta_{01}=(\beta,\alpha_{01})$ and $\alpha_{01}$ is the true vector of
non-zero coefficients.
\end{theorem}

 \noindent \textbf{Remark 3:}  As long as the postulated response and treatment model identify the true non-zero coefficients in each model as $n \rightarrow \infty$, the proposed variable selection method consistently identifies the set of non-ignorable confounders. Assuming linear working models, a sufficient but {\it{not}} necessary condition for selecting non-ignorable confounders is the linearity of the true models in their parameters. In \ref{app:sim}, we conducted simulation studies under different misspecification scenarios where the true models are non-linear in parameters and working models are linear.

\subsection{Choosing the Tuning Parameter} \label{tuning parameter}

We select the tuning parameter using the {\it Generalized Cross Validation}
(GCV) method suggested by \cite{tibshirani1996regression} and
\cite{fan2001variable}. Let $\bW=(D,\bX)$, then
\[
\text{GCV}(\lambda)=\frac{ \text{RSS}(\lambda)/n}{[1-d(\lambda)/n]^2},
\]
where $\text{RSS}(\lambda)=||Y-\bW \widehat \eta||^2$, $d(\lambda) =
\text{trace} [\bX(\bX^\transpose \bX + n\Sigma_{\lambda}(\widehat
\eta))^{-1})\bX^\transpose]$ is the effective number of parameters and
$\Sigma_{\lambda}( \eta) =
\text{diag}[p^\prime_{\lambda}(|\eta_1|)/|\eta_1|,...,p^\prime_{\lambda}(|\eta_{r_2}|)/|\eta_{r_2}|]$.   The selected tuning parameter $\widehat \lambda$ is defined by $\widehat
\lambda= \text{arg}\min_{\lambda} \text{GCV}(\lambda)$.

\subsection{Estimation}

In the treatment effect estimation, we fit the following model using the set of selected covariates in the previous step.  Note that a user may want to use other causal adjustment models such as IPTW or propensity score matching.

Our model is a slight modification of the conventional propensity score regression approach of \cite{RobinsMarkNewey1992}, and specifies
\begin{align}
\E[Y_i|S_i=s_i, \bX_i=\bx_i]=\theta s_i+ g(\bx;\gamma), \label{eq:PSR5}
\end{align}
where $S_i=D_i-\E[D_i|x_i] = D_i -\pi(\bx_i)$, $g(\bx;\gamma)$ is a function of
covariates and $\pi$ is the propensity score. The quantity $S_i$ is used in
place of $D_i$; if $D_i$ is used the fitted model may result in a biased estimator for $\theta$ since $g(x;\gamma)$ may be incorrectly specified. By defining $S_i$ in this way, we restore $\text{cor}[S_i,X_{ij}]=0$ for
$j=1,2,..,p$ where $p$ is the number of the selected variables (if $\pi(x_i) = \E[D_i|x_i]$ is correctly specified), as $\pi(x_i)$ is the (fitted) expected value of $D_i$, and hence $\bx_j^\transpose (D - \pi(x) )=0$, where $\bx_j^\transpose = (x_{1j},\ldots,x_{nj})$. Therefore,
misspecification of $g(.)$ will not result in an inconsistent estimator of $\theta$.

In general, this model results in a \textit{doubly robust} estimator (see \cite{davidian2005semiparametric}, \cite{Schafer2005} and \cite{bang2005doubly}); it yields a consistent estimator of $\theta$ if \textit{either} the propensity score model or conditional mean model \eqref{eq:PSR5} is correctly specified, and is the most efficient estimator (\cite{tsiatis2006semiparametric}) when both are correctly specified. For additional details on the related asymptotic and finite sample behavior, see \cite{kang2007demystifying}, \cite{neugebauer2005prefer}, \cite{van2003unified}
and \cite{robins1999robust}.

 The model chosen for estimation of the treatment effect is data dependent. Owing to the inherited uncertainty in the selected model, making statistical inference about the treatment effect becomes ``post-selection inference". Hence, inference about the treatment effect obtained in the estimation step needs to be done cautiously. The weak consistency of the estimator results from the following theorem.
\begin{theorem}
Let $\zeta(\hat \theta_{M_n},M_n)$ be a smooth function of $\hat \theta_{M_n}$ where $M_n$ is a set of selected variables using our method. Then $\zeta(\hat\theta_{M_n},M_n) \CiP \zeta(\theta_{M_0},M_0)$ as $n \rightarrow \infty$ where $M_0$ is the set of non-zero coefficients.
\label{th:consis}
\end{theorem}
Although, in this paper, we do not derive the asymptotic variance of the treatment effect estimator, in the simulation section, we provide some empirical results about the performance of a bootstrap estimator which is based on a method introduced by \cite{chatterjee2011bootstrapping}.

\subsection{The Procedure Summary}

The penalized treatment effect estimation process explained in sections 3.1 to 3.5 can be summarized as follows: \vspace{-.1in}
\begin{enumerate}
\item Estimate the vector of parameter $\hat \eta$ as $ \arg \sup_{\eta} \sum_{i=1}^n \log f_p(z_i; \eta)$ where $f_p(.)$ is defined in section 3.1.
\item Using the covariates with $\eta \neq 0$, fit a propensity score $\pi(\bX)$. 
\item Define a random variable $S_i=D_i-\pi(\bX_i)$ and fit the response model $\E[Y_i|d,\bx]=\theta s_i+ g(\bx_i;\gamma).$
The vector of parameters $(\theta,\gamma)$ is estimated using standard least square method. For simplicity, we assume the linear working model for $g(\bx_i;\gamma)=\gamma^\transpose\bx_i$. The design matrix $\bX$ includes a subset of variables with  $\eta \neq 0$.
\end{enumerate}


\section{Simulation Studies} \label{sec:simulation}
In this section, we study the performance of our proposed variable selection method using  simulated data when the number of covariates ($r_2$) is larger than the sample size. This also includes a scenario in which there is a weak non-ignorable confounder that is strongly related to the treatment but weakly to the outcome. We consider linear working models for both $g()$ and $h()$ functions throughout this section.

\begin{table}
\caption{\label{tab:largep} Performance of the proposed method when $r_2>n$ and in the presence of a weak confounder. S.D$^{emp}$: empirical standard error; S.D$^{tb}$: sandwich standard error.}\centering
\begin{tabular}{lrrrr|rrrr|} \hline

Method   &\multicolumn{1}{c}{Bias} & \multicolumn{1}{c}{S.D$^{emp}$} & \multicolumn{1}{c}{S.D$^{tb}$} & \multicolumn{1}{c|}{MSE} &
\multicolumn{1}{c}{Bias} & \multicolumn{1}{c}{S.D$^{emp}$} & \multicolumn{1}{c}{S.D$^{tb}$}& \multicolumn{1}{c}{MSE}  \\ \hline
Scenario 1.               & \multicolumn{3}{c}{$n=300$} & \multicolumn{3}{c}{$n=500$} \\
SCAD  & 0.010 &0.515 &0.502&0.266&0.012 &0.386 &0.381&0.149  \\
LASSO    &    0.067 &0.522&0.509&  0.277&0.057 &0.425&0.421 & 0.184          \\
PS-fit   &          0.164 &5.575& --&31.104&0.101&4.295 &--&18.453  \\
Oracle   &          0.017 &0.510& --& 0.260&0.007& 0.373 &--& 0.139 \\ \hline
Scenario 2.               & \multicolumn{3}{c}{$n=300$} & \multicolumn{3}{c}{$n=500$} \\
SCAD   &0.062 &0.606 &0.592& 0.372&0.019 &0.483  &0.456&0.234  \\
LASSO    &    0.037 &0.612& 0.593 &0.375&0.012 &0.481 & 0.460&0.232         \\
Y-fit   &         0.710 &0.598  &--&0.862&0.818&0.453 &-- &0.875 \\
PS-fit   &          0.381 &6.722 &--&45.326&0.094&5.117 & --&26.189  \\
Oracle   &         0.045 &0.638 &--& 0.409&0.018& 0.459  &--&0.211 \\ \hline
\end{tabular}
\end{table}

We generate 500 data sets of sizes 300 and 500 from the following two models:
\begin{enumerate}
\item[1.] $D \sim \text{Bernoulli}\left(\dfrac{\exp\{0.5x_1+0.5x_6-0.5x_7-0.5x_8
\}}{1+\exp\{0.5x_1+0.5x_6-0.5x_7-0.5x_8\}}\right)$ \\
$Y \sim \text{Normal}(d+2x_1+0.5x_2+5x_3+5x_4, 2)$
\item[2.] $D \sim \text{Bernoulli}\left(\dfrac{\exp\{0.5x_1+x_2+0.5x_6-0.5x_7-0.5x_8}
    {1+\exp\{0.5x_1+x_2+0.5x_6-0.5x_7-0.5x_8\}}\right)$, \\
$Y \sim \text{Normal}(d+2x_1+0.2x_2+5x_3+5x_4, 2)$
\end{enumerate}
where $\bX_k$ has a $N(1,2)$ for $k=1,...,550$. Note that in the second scenario, $x_2$ is considered as a weak confounder.  Results are summarized in Table \ref{tab:largep}; the {\it {Y-fit}} row refers to the estimator obtained by penalizing the outcome model using {\it SCAD} penalty.

We estimate the standard error of the treatment effect using an idea similar to \cite{chatterjee2011bootstrapping}. We bootstrap the sample and in each bootstrap force the components of the penalized estimator $\hat \eta$ to zero whenever they are close to zero and estimate the treatment effect using the selected covariates. More specifically, we define $\tilde \eta =\hat \eta I(|\hat \eta|> 1/\sqrt n )$. We utilize this thresholded bootstrap method to estimate the standard error of the treatment effect (S.D$^{tb}$). Although more investigation is required to validate the asymptotic behaviour of this method, our simulation results in Table \ref{tab:largep} show that the estimated standard error S.D$^{tb}$ is close to  the empirical estimator S.D$^{emp}$ (slightly underestimated).

 In the first scenario  there is no weak confounder and the  {\it {Y-fit}} is omitted since the result is similar to the SCAD row. The variance of the estimator in the {\it{PS-fit}} is too large due to the inclusion of  spurious variables that are not related to the response. The SCAD and LASSO estimators, however, are unbiased and perform as well as the oracle model. In the second scenario, the {\it {Y-fit}} estimator is bias because of under selecting the confounder $X_2$ while the proposed estimators using both SCAD and LASSO  remain unbiased.  Table \ref{tab:largep2} presents the average number of coefficients set to zero correctly or incorrectly under the second scenario. This, in fact, highlights the importance of our proposed method.

\begin{table} 
\caption{ {\small { Penalized ATE estimators based on the SCAD and LASSO penalty functions.  }}}\centering
\begin{tabular}{c c c  |c c c c c c} \hline
Method  & \multicolumn{1}{c}{Correct} & \multicolumn{1}{c}{Incorrect}
 & \multicolumn{1}{c}{Correct} & \multicolumn{1}{c}{Incorrect}\\
\hline
  & \multicolumn{2}{c}{$n=300$} & \multicolumn{2}{c}{$n=500$}\\
SCAD   &546 &0.05  &   546 &0.05    \\
LASSO &545 &0.00   &  546 &0.00   \\
 Y-fit      &546 &0.90   &   546 &0.92   \\           \hline
\end{tabular}
\label{tab:largep2}
\end{table}

In \ref{app:sim}, we examine the performance of our covariate selection estimation procedure when either of the working models of $g()$ or $h()$ is misspecified. Our results show that the proposed method outperforms both {\it {Y-fit}} and {\it {PS-fit}}.

\section{Application to Real Data } \label{sec:application}

In this section we examine the performance of our proposed method on the cross-country economic growth data used by \cite{doppelhofer2003determinants}. For   illustration purposes, we focus on a subset of the data which includes 88 countries and 35 variables. Additional details are provided in \cite{doppelhofer2009jointness}. We are interested in selecting non-ignorable variables which confound the effect of {\it life expectancy} (exposure variable) on the {\it average growth rate of gross domestic product per capita in 1960-1996} (outcome). 

The causal effect of life expectancy on economic growth is controversial. \cite{acemoglu2006disease} find no evidence of increasing the life expectancy on economic growth while \cite{husain2012alternative} shows that it might have positive effect. We dichotomize the life expectancy based on the observed median, which is 50 years.  Hence, the exposure  variable D=1 if life expectancy is below 50 years in that country and 0 otherwise.



\begin{table}
\caption{The economic growth data: List of significant variables. Penalized ATE estimators based on the SCAD and LASSO penalty functions. The two estimators PS-fit and Y-fit are obtained by penalizing the propensity score and outcome model via SCAD penalty, respectively. }\centering
\begin{tabular}{| l |rrrr|} \hline
  Variable   & Y-fit&PS-fit& SCAD& LASSO\\ \hline
 Ethnolinguistic Fractionalization &-0.39&-0.43& \textbf{-0.42}& \textbf{-0.33}\\
Population Density 1960 &-0.01 &0.00& \textbf{-0.16}& 0.00\\
 East Asian Dummy &0.48 &0.13&\textbf{0.53}& \textbf{0.45}\\
 Initial Income (Log GDP in 1960)    &0.00 &0.96&\textbf{0.19}& \textbf{0.15}\\ 
Public Education Spending Share    &0.05 &0.00&\textbf{0.13}&0.00\\ 
Nominal Government Share    &0.00 &0.00&\textbf{-0.18}&0.00  \\ 
Higher Education Enrolment &0.00 &0.23&0.00&0.00 \\ 
Investment Price    &-0.25 &0.00&\textbf{-0.24}& \textbf{-0.16} \\ 
Land Area Near Navigable Water    &0.00 &0.52&0.00&0.00\\ 
 Fraction GDP in Mining    &0.00 &0.00&\textbf{0.11}&0.00\\
 Fraction Muslim   &0.00 &-0.05&0.00&0.00\\
 Timing of Independence    &0.00 &-0.11&0.00&0.00\\
 Political Rights    &0.00 &-0.52&0.00&0.00\\
 Real Exchange Rate Distortions   &-0.04 &-0.04&\textbf{-0.20}&0.00\\
 Colony Dummy  &0.00 &-0.09&0.00&0.00\\
 European Dummy    &0.00 &0.00&\textbf{0.59}& \textbf{0.25}\\
 Latin American Dummy   &-0.18 &0.00&0.00&0.00\\
 Landlocked Country Dummy   &0.00 &0.00&\textbf{-0.21}&0.00\\ \hline
\end{tabular}
\label{tab:list}
\end{table}

We select the significant covariates  for the conditional mean and the treatment allocation models using the penalized likelihood (\ref{eq:Hahn}).  After covariate selection, we fit the model $\E[Y] = \theta s+ g(x;\gamma)$, where
$\theta$ is the treatment effect parameter (the function $g()$ assumed to be linear). Interaction or the higher order of the propensity score can be added to the response model if needed.

In our analysis, {\it{PS-fit}} and {\it{Y-fit}} refer to the cases where just the propensity score model and the conditional outcome models are penalized using SCAD to select the significant covariates (LASSO has a similar performance). 
Table \ref{tab:list} presents the list of variables and their estimated coefficients which are selected at least by one of the methods.

The proposed method selects 11 variables while {\it Y-fit} and {\it PS-fit} select 7 and 10 variables, respectively. This is mainly because of non-ignorable confounders which  either barely predict the outcome or treatment. More specifically, {\it Population Density 1960,   Initial Income, Public Education Spending Share,} and {\it Investment Price} are such non-ignorable confounders. 
Table \ref{tab:gdata} shows that although  the effect of life expectancy is positive, it is not significant.  Hence our results are consistent with \cite{acemoglu2006disease}. As we expected {\it PS-fit} results in inflating the standard error because  instrumental variables such as {\it Higher Education Enrollment}, {\it Land Area Near Navigable Water} and {\it Colony Dummy} are included. Also, including these variables in the propensity score causes bias. This is a confirmatory example of the result given by  \cite{de2011covariate} and \cite{abadie2006large}.

\begin{table}
\caption{The economic growth data: Penalized ATE estimators based on the SCAD and LASSO penalty functions. The two estimators PS-fit and Y-fit are obtained by penalizing the propensity score and outcome model via SCAD penalty, respectively.  }\centering
\begin{tabular}{|c|ccl|} \hline
  Method   & ATE & S.D.  & C.I.($\%95$) \\ \hline
SCAD &0.438 &0.405& (-0.372,1.248)\\
LASSO &0.451 &0.400&(-0.348,1.252) \\
Y-fit &0.394 &0.337& (-0.280,1.068) \\
PS-fit &0.774 &0.890&(-1.006,2.554)\\ \hline
\end{tabular}
\label{tab:gdata}
\end{table}

\section{Discussion} \label{sec:conclude}

We establish a two-step procedure for estimating the treatment effect in high-dimensional settings. First, we deal with the sparsity by penalizing a reparametrized conditional joint likelihood of the outcome and treatment given covariates. Then, the selected variables are used to form a double robust regression estimator of the treatment effect by incorporating the propensity score in the conditional expectation of the
response. The selected covariates may be used in other causal
techniques as well as the proposed regression method. 

Although, in high-dimensional cases, asymptotically penalizing the conditional outcome model given treatment and covariates is a valid variable selection approach in causal inference, it may perform poorly in finite sample by underselecting non-ignorable confounders which are weakly associated with outcome. Our proposed method improves the finite sample performance of the outcome penalization approach while maintaining the same asymptotic performance. The selected variables are used in a double robust regression estimator for estimating the treatment effect by
incorporating the propensity score in the conditional expectation of the
response.


 Any covariate selection procedure which involves the outcome variable affects the subsequent inference of the selected coefficients. This is because the selected model itself is stochastic and it needs to be accounted for. This is often referred to as ``post-selection inference'' in the statistical literature.   \cite{berk2012valid} proposes a method to produce a valid confidence interval for the coefficients of the selected model. In our setting, although we do not penalize the treatment effect, the randomness of the selected model affects the inference about the causal effect parameter through confounding. 
Moreover, note that the oracle property of the penalized regression estimators is a pointwise asymptotic feature and does not necessarily hold for all the points in the parameter space \citep{leeb2005model, leeb2008sparse}. In this manuscript, we assume that the parameter dimension ($r_2$) is fixed while the number of observation tends to infinity. One important extension to our work is to generalize the framework to cases where the tuple $(n,r_2)$ tends to infinity \citep{negahban2009unified}.  Analyzing the convergence of the estimated vector of parameters in the more general setting requires an adaptation of restricted eigenvalue condition \citep{bickel2009simultaneous} or restricted isometry property \citep{candes2007dantzig}.


\section*{Acknowledgment}
This research was supported in part by NIDA grant P50 DA010075. The second and third authors acknowledge the support of Discovery Grants from the Natural Sciences and Engineering Research Council (NSERC) of Canada. The authors are grateful to Professor Dylan Small for enlightening discussion.

\appendix

\makeatletter   

\newenvironment{app_item}{
\begin{itemize}
  \setlength{\itemsep}{0pt} \setlength{\parskip}{0pt}
  \setlength{\parsep}{0pt} \setlength{\topsep}{0pt}
  \setlength{\partopsep}{0pt}}{\end{itemize}}
 \makeatother

\section{Required conditions}
\label{app:conditions} In this Appendix, we prove the results stated in the text.
Here is the list of the regularity assumptions:
\begin{app_item}
\item C1. The parameter space $\Xi$ is a bounded open set in $\mathcal
    R^p$.
\item C2. The joint penalized density $f_p(z;\eta)$, where
    $z_i=(y_i,d_i,x_i)$ is a continuous function of $\eta$ on  $\Xi^c$ for
    almost all $z \in \mathcal Z$, where $\mathcal Z$ and $\Xi^c$ represent
    the sample space $(y_i,d_i,\bx_i)$ and the closure of $\Xi$ respectively.

\item C3. For all $\eta \in \Xi$ and all $\gamma>0$, $\kappa_{\eta} (\gamma)= \inf_{||\eta-\eta^*||>\gamma} r^2(\eta,\eta^*)>0,$
where $r^2(\eta,\eta^*)=\int_{\mathcal Z } [f^{1/2}(z;\eta)-f^{1/2}(z; \eta^{\ast})]^2 d\tau$.

\item C4. For $\eta \in \Xi^c$, $w_{\eta}(\delta)= \left[ \int_{\mathcal Z} \sup_{||h|| \leq \delta}\{  f^{1/2}(z;\eta)-f^{1/2}(z;\eta+h) \}^2d\tau\right] \rightarrow 0 \text{ as } \delta\rightarrow 0 .$
\item C5. $f(z;\eta)$ has finite Fisher information at each $\eta \in \Xi$.
\end{app_item}

Assumption $C3$ is  the identifiability condition, essentially requiring  that the distance
between the averaged densities over the response and the covariates for
two different values of the parameters $\eta$ and $\eta^*$ be
positive. Assumption $C4$ is referred to as the smoothness condition; it states that
the distance of the joint densities over $\eta$ and $\eta^*$ when  $\eta
\rightarrow \eta^*$ should approach zero as the sample size goes to infinity.


\section{Cases where $\alpha=0$ }
\label{app:measure}

Assume that $X$ is the only confounder/covariate. We conceptualize the following (true) Gaussian structural equation model: 
\begin{align*}
X&=\epsilon_1 \\
Z&=a_{12}X+\epsilon_2  \\
Y&=a_{13}X+a_{23}Z+\epsilon_3  \\
\end{align*}
where $(\epsilon_1,\epsilon_2,\epsilon_3)$ are generated  from a standard normal distribution.  
Since we are considering cases where $\alpha=0$, the penalty function can be ignored by assumption $P1$. Assume that the parameter $\beta$ in the reparametrized likelihood (3) is known and let $g(x,\alpha)=h(x,\alpha)=\alpha x$. Then by taking a derivative with respect to  $\alpha$ of the likelihood (3), $\alpha$  is defined as
\[
\alpha=\frac{cov(x,y)+cov(x,z)[1-\beta]}{cov(x,x)}
\]
Hence $\alpha=0$ iff 1) $cov(x,y)=cov(x,z)=0$, 2) $cov(x,y)=0$ \& $\beta=1$, or 3) $cov(x,y)+cov(x,z)[1-\beta]=0$. The latter is a drawback of our method, however, this particular data generating low has zero measure. Note that $cov(x,y)+cov(x,z)[1-\beta]=0$ implies that $a_{13}+a_{12}a_{23}+a_{12}[1-\beta]=0$. This is a hypersurface in the space of $(a_{13},a_{12},a_{23},\beta)$ and the set of distributions that satisfy this restriction has measure zero in $\mathbb{R}^4$.  

The same argument can be extended to the cases with more than one confounder. Then we have a union of a finite set of hypersurfaces.  Also, the same idea can be generalized to settings where variables are not normally distributed.

\section{ Lemmas}

\begin{lemma}
Let $Z_1,...,Z_n$ be independent and identically distributed with a density
$f(Z,\eta)$ that satisfies the conditions of C1-C4. If the
penalty function satisfies P3, then as $n \rightarrow \infty$
\begin{align}
R_n(\eta_2) =\prod_{i=1}^n \left[\frac{f_p(z_i;\eta_1,\eta_2)}{f_p(z_i;\eta_1,0)} \right]< 1, \hspace{.2in} \text{for } \eta_2\neq0.
\label{eq:inc}
\end{align}
\label{lem:inc}
\end{lemma}

\begin{proof}
$R_n(\eta_2)$ can be written as

\[
\prod_{i=1}^n \left[\frac{f(z_i;\eta_1,\eta_2)e^{-\sum_{j=s}^p p_{\lambda_n}(|\eta_j|)}}{f(z_i;\eta_1,0)} \right].
\]
By theorem 1.1 in Chapter II of Ibragimov \& Has' Minskii (1981), it can be written as

\begin{align*}
R_n(\eta_2) = \exp\left [ \left\{\sum_{i=1}^n \frac{\partial \ln f_p(z_i;\eta_1,0)}{\partial \eta_2} \right\}||\eta_2||
-n\sum_{j=s}^p p'_{\lambda_n}(|\eta_j|)-\frac{1}{2} \eta_2I(\eta_1,0)\eta_2+\psi_n(\eta_2) \right],
\end{align*}
where $p(|\psi_n(\eta_2)|>\epsilon) \rightarrow 0$. Since $\sum _{i=1}^n
\partial \ln f(z_i;\eta_1,0)/\partial \eta_2=O_p(n)$, equivalent to the condition
$P3$, the desired inequality holds if
\[
\sum_{j=s}^p p'_{\lambda_n}(|\eta_j|) >||\eta_2|| = O_p(1) \qed
\]
Not that in our setting, $p_{\lambda_n}(|\eta_j|)= \frac{1}{|\eta_j|}p^*_{\lambda_n}(|\eta_j|)=\frac{\sqrt n}{O_p(1)}p^*_{\lambda_n}(|\eta_j|) $ where $p^*_{\lambda_n}(.)$ is one of the standard penalty functions such as LASSO or SCAD.
\end{proof}

The following Lemma is an adaptation of the results given by
Ibragimov \& Has' Minskii (1981), page 36.
\begin{lemma}
Suppose assumption C1-C4 are satisfied. Then for any fixed $\eta \in \Xi$ {\small{
\begin{align}
\E_{\eta} \left[ \sup_{\Gamma} \prod_{i=1}^n \frac{f_p^{1/2}(z_i;\eta+b)}{f_p^{1/2}(z_i;\eta)}   \right]
\leq \exp\left[ -\frac{n}{2}\left\{ \kappa_{\eta,n}(\frac{\gamma}{2}) -  2w_{\eta+b_0,n}(\delta) +
p_{\lambda_n}(|\eta+b^m|)-p_{\lambda_n}(|\eta|)  \right \}\right],
\label{eq:hesmin}
\end{align} }}\\
where $\Gamma$ is the sphere of radius $\delta$, situated in its entirely in
the region $||b||>\gamma/2$, $b_0$ is the center of $\Gamma$ and $\inf_{\Gamma}
p_{\lambda_n}(|\eta+b|)= p_{\lambda_n}(|\eta+b^m|)$. \label{lem:hesmin}
\end{lemma}

\begin{proof}
The proof follows from the proof of Theorem 1.4.3 in Ibragimov \& Has' Minskii (1981).  Let
\[
R_n(b) = \prod_{i=1}^n \frac{f_p(z_i;\eta+b)}{f_p(z_i;\eta)}=
\prod_{i=1}^n \frac{f(z_i;\eta+b)e^{- p_{\lambda_n}(|\eta+b|)}}{f(z_i;\eta)e^{- p_{\lambda_n}(|\eta|)}}.
\]
We want to find an upper bound for the expectation $\E_{\eta} \left[
\sup_{\Gamma} R_n^{1/2}(b) \right]$ , where $\Gamma$ is the sphere of a radius
$\delta$ situated in its entirety in the region $||b||>\frac{1}{2}\gamma$. If
$b_0$ is the center of $\Gamma$, then
\begin{align*}
\sup_{\Gamma}R_n^{1/2}(b) &= \sup_{\Gamma} \prod_{i=1}^n \left[\frac{f(z_i;\eta+b)e^{- p_{\lambda_n}(|\eta+b|)}}{f(z_i;\eta)e^{- p_{\lambda_n}(|\eta|)}} \right]^{1/2} \leq  \prod_{i=1}^n \sup_{\Gamma} e^{ \frac{1}{2}p_{\lambda_n}(|\eta|)-\frac{1}{2}p_{\lambda_n}(|\eta+b|)} \\
& \prod_{i=1}^n f^{-1/2}(z_i;\eta) \left[ f^{1/2}(z_i;\eta+b_0) +\sup_{h\leq \delta} |  f^{1/2}(z_i;\eta+b_0+h)-f^{1/2}(z_i;\eta+b_0)    |\right].
\end{align*}
Thus,

\begin{align*}
\E_{\beta} \left[\sup_{\Gamma}R_n^{1/2}(b)\right] &\leq \prod_{i=1}^n \sup_{\Gamma} e^{ \frac{1}{2}p_{\lambda_n}(|\eta|)-\frac{1}{2}p_{\lambda_n}(|\eta+b|)} \left[ \int_{\mathcal Z} f^{1/2}(z_i;\eta)f^{1/2}(z_i;\eta+b_0) d\tau \right. \\
& \left.+ \int_{\mathcal Z} \sup_{|h|\leq \delta} f^{1/2}(z_i;\eta)| f^{1/2}(z_i;\eta+b_0+h)-f^{1/2}(z_i;\eta+b_0)    | )^n d\tau \right].
\end{align*}
We further note that

\begin{align}
\label{eq:lemproof1}
\int_{\mathcal Z} f^{1/2}(z;\eta)f^{1/2}(z;\eta+b_0)d\tau = &\frac{1}{2}  \left[\int_{\mathcal Z}f(z;\eta)d\tau+\int_{\mathcal Z}f(z;\eta+b_0)d\tau \right.\\      \nonumber
&\left. -\int_{\mathcal Z}[f^{1/2}(z;\eta)-f^{1/2}(z;\eta+b_0)]^2d\tau  \right] \\  \nonumber
& \leq 1-\frac{1}{2}r^2(\eta+b_0) \leq 1 - \frac{\kappa_{\eta}(\frac{\gamma}{2})}{2}
\end{align} \vspace{-.7in} \\
and

\begin{align}
\int \sup_{|h|\leq \delta} f^{1/2}(z_i;\eta)| f^{1/2}(z_i;\eta+b_0+h)-f^{1/2}(z_i;\eta+b_0)    |d\tau \leq w_{b_0}(\delta).
\label{eq:lemproof2}
\end{align}
The last inequality follows from the Cauchy-Schwarz inequality.
Finally,
\begin{align*} \E_{\beta} \left[ \sup_{\Gamma}R_n^{1/2}(b) \right] \leq
\exp\left[-\frac{n}{2}\left\{ \kappa_{\eta}(\frac{\gamma}{2})-2w_{b_0}(\delta)
+p_{\lambda_n}(|\eta+b^m|)-p_{\lambda_n}(|\eta|)\right\} \right]
\end{align*}
where $\sup_{\Gamma} e^{- p_{\lambda_n}(|\eta+b|)}=
e^{-p_{\lambda_n}(|\eta+b^m|)}$, using the inequality $1+a \leq e^a$. \qed
\end{proof}


\section{Proofs of Theorem \ref{th:norm} \& \ref{th:consis} }
\textbf{\it{Proof of Theorem \ref{th:norm}}}: 
\textbf{Part (a)} Consider $\eta_0=(\eta_{01},0)$ and partition $\eta=(\eta_1,\eta_2)$. We need to show that in the neighbourhood $||\eta-\eta_0||<O( h_n)$ where $h_n \rightarrow 0$ as $n \rightarrow \infty$,
\[
\prod_{i=1}^n \frac{f_p(z_i;\eta_1,\eta_2)}{f_p(z_i;\widehat \eta_1,0)} <1.
\]
It can be written as

\[
\prod_{i=1}^n \frac{f_p(z_i;\eta_1,\eta_2)}{f_p(z_i;\widehat \eta_1,0)} = \prod_{i=1}^n \left[\frac{f_p(z_i;\eta_1,\eta_2)}{f_p(z_i; \eta_1,0)} \right]\left[\frac{f_p(z_i;\eta_1,0)}{f_p(z_i;\widehat \eta_1,0)} \right]<\prod_{i=1}^n \frac{f_p(z_i;\eta_1,\eta_2)}{f_p(z_i; \eta_1,0)} <1.
\]
By the result of Lemma \ref{lem:inc}, the last inequality holds with
probability one as $n \rightarrow \infty$.


\textbf{\it{Part (b) }}: Under the conditions listed in the Theorem, we have
\begin{align*}
\sum_{i=1}^n \log f_p(z_i; \eta_{01}+\frac{c}{\sqrt n}) - \log f_p(z_i; \eta_{01}) =& \frac{1}{\sqrt n}  c' \sum_{i=1}^n \frac{\partial \log f(z_i; \eta_{01})}{\partial \eta_{01}} - c' \sqrt np'_{\lambda_n}(\alpha_{01}) \\
&-c' p''_{\lambda_n}(\alpha_{01}) c     - \frac{1}{2} c'I(\eta_{01})c + R_n(\eta_{01},c), \hspace{.1in} \text{ for $|c|<M$},
\end{align*}
where $f_p(.)$ is the penalized density defined in $\S$4.2 and $c$ is a constant vector. Note that $\eta_{01}=(\beta,\alpha_{01})$ and $\alpha_{01}$ is the true vector of
non-zero coefficients. Using the proof of Theorem 2.1.1 in \cite{ibragimov1981statistical}, one can show that $R_n(\eta_{01},c)\rightarrow 0$ in probability.  

Using the proof of Theorem 2.5.2 in \cite{MR1245941},  we can show that for any $\epsilon>0$
\[
P\left( \left|\sqrt n (\hat \eta_{01}-\eta_{01})- \frac{1}{\sqrt n} I^{-1}(\eta_{01}) \sum_{i=1}^n \frac{\partial \log f(z_i; \eta_{01})}{\partial \eta_{01}} + \sqrt np'_{\lambda_n}(\alpha_{01}) + p''_{\lambda_n}(\alpha_{01})    \right| >\epsilon \right) \rightarrow 0,
\]
as $n \rightarrow \infty$. Under assumption P2, it completes the proof of part (b). \qed

\textbf{\it{Proof of Theorem \ref{th:consis}}}: 
Using the triangle inequality,
\begin{align*}
|\zeta(\hat \theta_{M_n},M_n)-\zeta(\theta_{0},M_0)| \leq |\zeta(\hat \theta_{M_n},M_n)-\zeta(\hat \theta_{M_0},M_0)|+|\zeta(\hat \theta_{M_0},M_0)-\zeta( \theta_{0},M_0)|.
\end{align*}
By differentiability of the $\zeta(.,.)$ function in $\theta$, we have $\zeta(\hat \theta_{M_0},M_0) \CiP \zeta(\theta_{0},M_0)$. Also, $\forall t>0$, we have
\begin{align*}
 p(|\zeta(\hat \theta_{M_n},M_n)-\zeta(\hat \theta_{M_0},M_0)|>t) &\leq p(\{M_n=M_0\} \cap \{|\zeta(\hat \theta_{M_n},M_n)-\zeta(\hat \theta_{M_0},M_0)|>t\}) \\
 &+ p(\{M_n \neq M_0\} \cap \{|\zeta(\hat \theta_{M_n},M_n)-\zeta(\hat \theta_{M_0},M_0)|>t\})\\
 &\leq p(M_n \neq M_0) = 0
\end{align*}
The last inequality follows by the oracle property of our procedure (Theorem \ref{th:consis}). See also Theorem 4.2 in \cite{wasserman2009high}.  This completes the proof of weak consistency.

\section{Existence of the consistent penalized maximum likelihood estimator }
\label{app:PMLEexists}

\begin{theorem}
Under assumptions C1-C4 and P1-P3, the penalized maximum pseudo-likelihood estimator $\widehat \eta_{n}$ converges to $\eta_0$ as $n \rightarrow \infty$ almost surely where $\eta_0$ is the true parameter value with respect to (3).
\label{th:PMLEexists}
\end{theorem}

\begin{proof}
For fixed $\gamma>0$, the
exterior of the sphere $\|\eta-\eta_0\|=\|b\|\leq \gamma $ can be covered by $N$
spheres $\Gamma_k$, $k=1,...,N$ of radius $\delta$ with centers $b_k$. The
small value $\delta$ is chosen such that (i) all the $N$ spheres are located in the
$||b||>\gamma/2$, (ii) $2w_{b_k}(\delta) \leq \kappa_{\eta}(\gamma/2)/4$ where $w()$ and $\kappa()$ are defined in $C3$ and $C4$, respectively, and (iii) $\forall b \in \Gamma_k$, $|p_{\lambda_n}(|\eta+b_k|)-p_{\lambda_n}(|\eta|)| \leq \kappa_{\eta}(\gamma/2)/4$. Let $u_k \in \Gamma_k$ so that $R(\hat u_k) = \sup_{u_k \in \Gamma_k} R(u_k)$. Then, in view of
the result of Lemma \ref{lem:hesmin}, we have
\begin{align*}
P(|\widehat \eta_{n} -\eta_0|>\gamma) &\leq \sum_{k=1}^N P(|\widehat \eta_{n} -\eta_0| \in \Gamma_k)  \leq \sum_{k=1}^N  P(\sup_{ u_k \in\Gamma_k} R_n( u_k)\geq R(0))  \\
& \leq \sum_{k=1}^N \exp \left[ -\frac{n}{2} \left\{ \kappa_{\eta}(\frac{\gamma}{2}) - 2w_{b_k}(\delta)  +p_{\lambda_n}(|\eta+b_k^m|)-p_{\lambda_n}(|\eta|)\right\} \right] \\
 & \leq N \exp \left[ -\frac{n}{2}\left\{ \kappa_{\eta}(\frac{\gamma}{2}) - \frac{1}{4}\kappa_{\eta}(\frac{\gamma}{2})  - \frac{1}{4}\kappa_{\eta}(\frac{\gamma}{2})\right\} \right],\\
 & \leq N \exp \left[ -\frac{n}{2}\left\{ \frac{1}{2}\kappa_{\eta}(\frac{\gamma}{2})\right\} \right],
\end{align*}
where $\sup_{b \in \Gamma_k} e^{- p_{\lambda_n}(|\eta+b|)}=e^{-p_{\lambda_n}(|\eta+b_k^m|)}$. Note that $R(0)=1$. The second inequality follows from the fact that when the MPL estimator $\hat \eta_n$ falls in at least one of the spheres $\Gamma_k$ where $\Gamma_k$ covers  outside of the neighborhood $\gamma/2$ of $\eta_0$, it means $\sup_{u_k\in \Gamma_k} \prod_{i=1}^n f_p(z_i;\eta_0+u_k)  \geq \prod_{i=1}^n f_p(z_i;\eta_0)$. By the definition of $R_n(u)$, this inequality can be written as $\sup_{ u_k \in\Gamma_k} R_n( u_k)\geq 1$.
Also the third inequality follows from Lemma S1 and
Markov's inequality.

Thus,

\[
P(|\widehat \eta_{n} -\eta_0|>\gamma) \leq N \exp \left[ -\frac{n}{4} \kappa_{\eta}(\frac{\gamma}{2}) \right],
\]
and hence we have strong consistency, as
\[
P\left( \bigcup_{m=n}^{\infty} |\widehat \eta_{2m}| \right) \leq
\frac{N \exp \left[ -\frac{n}{4} \kappa_{\eta}(\frac{\gamma}{2}) \right]}
{1- \exp \left[ -\frac{1}{4} \kappa_{\eta}(\frac{\gamma}{2}) \right]} \rightarrow 0 \text{ as } n\rightarrow \infty.
\qed
\]
\end{proof}

\section{Performance under model misspecifications}
\label{app:sim}
In this simulation study, we want to examine the performance of our proposed method when 1) either of the working models $g()$ or $h()$ are misspecified and 2) the number of potential confounders ($r_2$) is larger then the sample size.

\begin{enumerate}
\item[1.] $D \sim \text{Bernoulli}\left(\frac{\exp\{0.1x_1+x_2+0.7 \frac{x_{10}+x_{9}}{1+|x_8|}\}}{1+\exp\{0.1x_1+x_2+0.7 \frac{x_{10}+x_{9}}{1+|x_8|}\}}\right)$ \\ 
$Y \sim \text{Normal}(d+0.5x_1+0.1x_2+2x_3+2x_4, 2)$
\item[2.] $D \sim \text{Bernoulli}\left(\frac{\exp\{x_1-x_2-0.1x_{8}-x_{9}+x_{10}\}}{1+\exp\{x_1-x_2-0.1x_{8}-x_{9}+x_{10}\}}\right)$ \\ 
$Y \sim \text{Normal}(d+2x_8+2 \frac{\exp\{0.2x_3+0.2x_4\}}{\exp\{0.2|x_1|+0.2|x_2|\}}, 2)$
\end{enumerate}
where $\bX_k$ has a $N(0,2)$ for $k=1,...,550$.

In both scenarios, we consider linear working models for $g()$ and $h()$. Thus, at least one of them is misspecified. Table \ref{tab:miss} summarized the results. \textit{PS-fit} refers to the propensity score model including only the variables affecting treatment allocation (commonly done by practitioners) and {\it {Y-fit}} refers to the estimator obtained by penalizing the outcome model using SCAD penalty. We applied our variable selection procedure using the SCAD and LASSO penalties.   

In scenario 1, $x_2$ is a non-ignorable confounder which is weakly associated with the outcome. Ignoring this variable by {\it {Y-fit}} method results in bias which does not go zero by increasing the sample size.  \textit{PS-fit} method, in scenarios 1 \& 2, ignores  the non-ignorable confounders $x_1$ and $x_8$, respectively, which leads to a bias treatment effect estimate. Our proposed method using SCAD and LASSO outperforms all the other methods by increasing the chance of including all the confounders (weak or strong) in the model.

\begin{table}
 \caption{\label{tab:miss} Performance of the proposed method when either of the response or treatment models are misspecified and $r_2>n$.}\centering
\begin{tabular}{lrrr|rrr|} \hline
Method   &\multicolumn{1}{c}{Bias} & \multicolumn{1}{c}{S.D.} & \multicolumn{1}{c|}{MSE} &
\multicolumn{1}{c}{Bias} & \multicolumn{1}{c}{S.D.} & \multicolumn{1}{c}{MSE}  \\ \hline
Scenario 1.               & \multicolumn{3}{c}{$n=300$} & \multicolumn{3}{c}{$n=500$} \\
SCAD  & 0.036& 0.453 &0.206&   0.047 &0.309 &0.097 \\
LASSO    &  0.209  &0.456 &0.252&0.057 &0.425&  0.184          \\
Y-fit   &         0.109 &0.360 &0.142&0.121& 0.290 &0.099 \\
PS-fit   &          0.227 &1.033 &1.118&  0.157 &0.805 &0.665 \\ \hline
Scenario 2.               & \multicolumn{3}{c}{$n=300$} & \multicolumn{3}{c}{$n=500$} \\
SCAD       &0.116 &0.358 &0.142&0.091 &0.319 &0.110 \\
LASSO     &    0.110 & 0.377 &0.154 &0.016 &0.288 &0.083         \\
Y-fit    &         0.185 &0.398 &0.193&0.205 &0.291 &0.126 \\
PS-fit   &          0.737 &0.918 &1.387& 0.673& 0.768 &1.044\\\hline 
\end{tabular}
\end{table}

\bibliographystyle{Biometrika}
\bibliography{mybib-1}

\end{document}